\documentclass[BCOR8mm,DIV14]{scrartcl}
\usepackage{bbm}
\usepackage[english]{babel}
\usepackage[all]{xy}
\usepackage{amsmath,amssymb,amsthm}
\theoremstyle{definition}
\newtheorem{theorem*}{Theorem}
\newtheorem{theorem}{Theorem}

\newtheorem{corollary}{Corollary}

\begin{document}
\title{H\"older Functionals and Quotients}
%\thanks{Thanks to Stefan Th\"urey who read the text carefully} 
\author{VOLKER W. TH\"UREY}                  
\maketitle
\begin{abstract}   \centerline{Abstract}  
       We describe an inequality of finite or infinite sequences of real numbers and their quotients. \
       More precisely, we compare the  quotient of H\"older functionals of two sequences of numbers with
       the sum of their quotients. In the last section we investigate the `wideness' of the inequality,
       i.e. we show that  both the inequality can converge into an equality, and the difference between 
       the two sides of the inequality can be arbitrary large.   
  %     introduce a related function which relies on the  dicussed inequality.
\end{abstract}
%%%%% END OF TITLE PAGE %%%%%%%%%%%%%%%%%%%%%%%%%%%%%%%%%%%%%%%%%%%%%%
%  \footnote{      {\it Keywords and phrases}: elementary inequality }   
%  \footnote{2010  {\it Mathematical Subject Classification}: 26D15 }
%%%%% BODY OF THE PAPER %%%%%%%%%%%%%%%%%%%%%%%%%%%%%%%%%%%%%%%%%%%%%%
% You should eventually delete (after reading) the rest of the text below %%
 \section{Introduction}
   For convenience, we restrict our considerations on positive real numbers, i.e. 
  $$   a_1, a_2, a_3, \ \ldots \ \ \  b_1, b_2, b_3, \ \ldots  \ \ > 0 \ , $$  
     to avoid problems  with a denominator $0$, otherwise we have to discuss cases with expressions like 
     `$\frac{0}{0}$' and  `$\frac{a}{0}$'.                       
     Further, it is easy to extend the coming theorems on negative numbers by using the modulus of a number.  
      
     We start with a known statement wich is called the `Rearrangement Inequality'.  
\begin{theorem}     \label{thefirstone}   \it
        Let us take two finite ordered sequences of positive real numbers of the same length, i.e. we have  
        $$   0 < a_1 \leq  a_2 \leq a_3 \leq \ \ldots \ \leq a_n \ \  \text{ and } \ \ 
                             0 < b_1 \leq b_2 \leq b_3 \leq \ \ldots \ \leq b_n  $$ 
        for a natural number $n$. \ Let \ $ \sigma $ be any permutation on the set 
        \ $ \{1,2,3, \ldots , n \} $.  We have the inequality of sums of fractions
   $$         \sum_{k=1}^{n} \ \frac{a_k}{b_k} \ \ \leq \ \
              \sum_{k=1}^{n} \ \frac{a_k}{b_{\sigma(k)}} \ \ \leq \ \ \sum_{k=1}^{n} \ \frac{a_k}{b_{n-k+1}} \ . $$
   %     This is the main theorem, not proved in \cite{AS}. And this is the main formula
   %      \begin{equation}
   %     \label{important}
   %     2+2=4.
   %     \end{equation}
   %     Formula \eqref{important} seems to be new.
\end{theorem}
\begin{proof}
         We can find it in \cite{HLP}.   %   Theorem 368.    %%%%%%%%%%%%  or \cite{Mit}.   
  %      At the end of the proof, use command qed if you want to obtain 
  %      square simbol. Using amsthm package the square simbol will be inserted automatically.
\end{proof}
    We say \ $ \vec{a} $ \ for any n-tuple  $( a_1,a_2,a_3, \ldots , a_n )$ of positive real numbers.  \\
        We fix any non-zero real number \ $p$. \     Note that the following expression
    $$  \|\vec{a}\|_p := \sqrt[p]{a_1^{p} + \ a_2^{p} + \ a_3^{p} + \ \ldots \ + \ a_n^{p} } $$
       is defined for all \ $p\neq 0$, since all $ a_k $ are positive. We call 
       $  \|\vec{a}\|_p $ a { \it H\"older functional} of the tuple \ $ \vec{a} = (a_1,a_2, \ldots , a_n)$.      
         
   Now we are prepared for the main theorem.       
 \begin{theorem}[Main inequality]      \label{thesecondone}   \it
          Let us take  H\"older functionals of two finite sequences of positive real numbers of the same
          length  $ n \in {\mathbb N}, \ n \geq 2 $. \ That means we have \ \ 
          $ 0 < a_1, a_2, a_3, \ \ldots \ , \ a_n$ \ \  and $ 0 < b_1, b_2, b_3, \ \ldots \ , \ b_n $.  \ \ 
          For all real numbers \ $p \neq 0$ \ there is the following strict inequality
   $$     \frac{\|\vec{a}\|_p}{\|\vec{b}\|_p} \ 
          = \  \sqrt[p]{ \frac{ a_1^{p} + \ a_2^{p} + \ a_3^{p} + \ \ldots \ + \ a_n^{p} }
          { b_1^{p} + \ b_2^{p} + \ b_3^{p} + \ \ldots \ + \ b_n^{p} } } \ \ <  \ \ 
          \sum_{k=1}^{n} \ \frac{a_k}{b_k} \  .                             $$
   %                    \ \ \leq \ \ \sum_{k=1}^{n} \ \frac{a_k}{b_{n-k+1}} \ . $$
     The inequality remains valid also for the limits \ $ p = -\infty, \ p= \infty, \ \text{and} \ p=0 $.  
          Further, the inequality is sharp, i.e. in this generality it can not be improved. 
          Further, for $ n = 1 $ we have a trivial equality.
\end{theorem}    
%\begin{remark}
        Note that the arrangement of the numbers \
        $ a_1, a_2, a_3, \ \ldots \ , \ a_n$ \ and \ $ b_1, b_2, b_3, \ \ldots \ , \ b_n $ \ 
        does not affect the left hand side of the inequality but the right hand side. If the inequality is true,
        it must be true in the `worst' case, i.e. for the arrangement 
    $$  0 < a_1 \leq a_2 \leq a_3 \leq \ldots \leq a_n \ \ \ \text{and} \ \ \  
                             0 < b_1 \leq b_2 \leq b_3 \leq \ldots \leq b_n \ ,  $$  
        see Theorem  \eqref{thefirstone}. By this arrangement the right hand side is as small as possible.
    %    Let \ $ \sigma $ be any permutation on the set  \ $ \{1,2,3, \ldots , n \} $. 
    %     Without restriction of   generality  we fix an order \\
%\end{remark}   
                           
 In the next section we prove the theorem, which is the main contribution of this article.
        In the last section we show that the difference between both sides of the inequality can become 
        arbitrary  small and arbitrary big. 
  %      deals with a real valued function. This function is constructed directly from the properties of Theorem  \eqref{thesecondone}.
%%%%%%%%%%%%%%%%%%%%%%%%%%%%%%%%%%%%%%%%%%%%%%%%%%%%%%%%%%%%%%%%%%%%%%%%%%%%%%%%%%%%%%%%%%%%%%%%%%%%%
 \section{The Proof of the Inequality}       
 \begin{proof}  
        The way of proving the theorem is not surprising.  %   For \ $n=1$ \ it is trivial. 
        We prove it for  $n=2$, which is the harder part,   and after that we go the induction step from \ 
        $ n $ to $ n+1$. 
                                                                                                             
   \underline{Beginning of the induction:} \quad Let $ n = 2 $.  
            Let us take positive numbers \ $ a_1, a_2, b_1, b_2 $ with the arrangement \
            $ 0 < a_1 \leq a_2 $ \ and \ $ 0 < b_1 \leq b_2 $.     We want to prove that
   \begin{align}      \label{ungleichung eins}         
       \sqrt[p]{\frac{a_1^{p} + \ a_2^{p}}{b_1^{p} + \ b_2^{p}}} \  < \ \frac{a_1}{b_1} + \frac{a_2}{b_2}    
   \end{align}
            holds for all real numbers \ $ p \neq 0 $.  \ Since \ $ a_1 \leq a_2 $ \ and \ $ b_1 \leq b_2 $ \    
            there are two positive numbers \ $ 0 <  \alpha, \beta \leq 1 $ \ with \ $ a_1 = \alpha \cdot a_2 $ \ and \
             $ b_1 = \beta \cdot b_2 $. \ Inequality \eqref{ungleichung eins} is equivalent to
   \begin{align}     \label{ungleichung zwei}           
         \sqrt[p]{\frac{\alpha^{p} + 1}{\beta^{p} + 1}} \  < \  \frac{\alpha}{\beta} +  \frac{1}{1}                                                                            \  =  \  \frac{\alpha}{\beta}  +   1 \ .
    \end{align}     
    We distinguish four cases. The Case {\bf D} deals with a negative \ $ p $. 
     \begin{itemize}   
               \item Case {\bf A}: \ \ $  0 < p $ \ and \  $ \alpha  \leq \beta $, 
               \item Case {\bf B}: \ \ $  0 < p \leq 1 $ \ and \  $ \beta < \alpha $, 
               \item Case {\bf C}: \ \ $  1 < p$  \ and \  $ \beta < \alpha $,   \nopagebreak[40]
               \item Case {\bf D}: \ \ $  p < 0$.
     \end{itemize}  
    \underline{Case {\bf A}}: We assume \ $ 0 < p $ \ and \  $ \alpha  \leq \beta $. We have  
        \begin{align}     \label{ungleichung drei}           
              \sqrt[p]{\frac{\alpha^{p} + 1}{\beta^{p} + 1}} \ \leq \ \sqrt[p]{\frac{\alpha^{p} + 1}{\alpha^{p} + 1}} \
              = \ 1 \ < \ \  \frac{\alpha}{\beta}  +   1 \ .  
        \end{align}  
     \underline{Case {\bf B}}: We assume \ $  0 < p \leq 1 $ \ and \  $ \beta < \alpha $. \ We set \ 
         $ q := \frac{1}{p} $, \ hence \ $ 1 \leq q $. \ We want to prove Inequality  \eqref{ungleichung zwei},
         we write it again as    
    \begin{align}     \label{ungleichung vier}           
         \sqrt[p]{\frac{\alpha^{p} + 1}{\beta^{p} + 1}} \ = \ 
         \left[ \frac{ \sqrt[q]{\alpha} + 1}{  \sqrt[q]{\beta} + 1} \right]^{q}
         \  < \  \frac{\alpha}{\beta}  +   1 \ .
    \end{align}     
    We have the following chain of equivalences to the desired Inequality \eqref{ungleichung zwei}  
    \begin{align*}   
      %     \sqrt[p]{\frac{\alpha^{p} + 1}{\beta^{p} + 1}} \  < \ \frac{\alpha}{\beta} +  1 \ \       
        \left[ \frac{ \sqrt[q]{\alpha} + 1}{  \sqrt[q]{\beta} + 1} \right]^{q}
        \  < \  \frac{\alpha}{\beta}  +  1 \ \
        \Longleftrightarrow  \ \ &  \left[ \sqrt[q]{\alpha} + 1 \right]^{q}
        \  < \ \left[ \sqrt[q]{\beta} + 1 \right]^{q} \cdot \left[ \frac{\alpha}{\beta}  +  1 \right] \\
        \Longleftrightarrow  \ \ &  \sqrt[q]{\alpha} + 1 
        \  < \ \left[ \sqrt[q]{\beta} + 1 \right] \cdot \sqrt[q]{ \frac{\alpha}{\beta}  +  1 } \\
        \Longleftrightarrow  \ \ &  \sqrt[q]{\alpha} + 1
        \  < \ \sqrt[q]{\alpha + \beta} \ + \ \sqrt[q]{ \frac{\alpha}{\beta}  +  1 } \ .
   \end{align*}  
   The last inequality is obvious, which finishes  Case {\bf B}.      \\   
   \underline{Case {\bf C}}: We assume \ $  1 < p $ \ and \  $ \beta < \alpha $. \
    We can write the chain of inequalities    
     \begin{align*}   
         \sqrt[p]{\frac{\alpha^{p} + 1}{\beta^{p} + 1}} \ < \ 
         \sqrt[p]{\frac{\alpha^{p} + 1}{1}} \ \leq \ \sqrt[p]{2} \ < \ 2 \ < \  \frac{\alpha}{\beta} +  1 \: ,               \end{align*} 
     and  Case {\bf C} is proven. \ Therefore Inequality \eqref{ungleichung eins} is shown for all real $ p > 0 $.  \\
    \underline{Case {\bf D}}: We investigate the case of a negative real number $ p $, i.e. let \ $ p < 0 $. 
    We define \ $ q := -p $, \ i.e. $ q $ is a positive number, hence we can refer to the first three cases.  We write
    \begin{align*}   
         \sqrt[p]{\frac{\alpha^{p} + 1}{\beta^{p} + 1}} \ = \ \sqrt[-q]{\frac{\alpha^{-q} + 1}{\beta^{-q} + 1}} \ = \ 
         \sqrt[q]{\frac{\left[\frac{1}{\beta}\right]^{q}+1}{\left[\frac{1}{\alpha}\right]^{q}+1}} \ < \ \ 
         \frac{\frac{1}{\beta}}{\frac{1}{\alpha}} + \frac{1}{1} \ = \  \frac{\alpha}{\beta} + 1 \ .
    \end{align*} 
    This shows Case {\bf D}, and the last of four cases to prove the beginning of the 
    induction with \ $n=2$ \ is done.   
                  
    \underline{Induction step:} \quad Let the theorem be proven for a natural number $ n \geq 2 $.  We prove it for
    the next  number \ $ n + 1 $.  Let us take two sets of $ n+1 $ positive numbers, i.e. we assume \
    $ a_1, a_2, a_3, \ \ldots \ , \  a_n, a_{n+1} \ \ \text{and} \ \ 
     b_1, b_2, b_3, \ \ldots \ , \ b_n, b_{n+1} \ > \ 0 \ . $    \ We have  
  \begin{align*}       
          \  \sqrt[p] { \frac{ a_1^{p} + a_2^{p} + \ \ldots \ + a_n^{p} + a_{n+1}^{p} }
                              { b_1^{p} + b_2^{p} + \ \ldots \ + b_n^{p} + b_{n+1}^{p}} } \ \ = \ & \ 
          \sqrt[p]{ \frac{  \left( \sum_{k=1}^{n-1} a_k^{p} \right) \ + \ a_{n}^{p} \ + \ a_{n+1}^{p}}
                         {  \left( \sum_{k=1}^{n-1} b_k^{p} \right) \ + \ b_{n}^{p} \ + \ b_{n+1}^{p}}  }   \\  
           \ \ = \ & \    \sqrt[p]{ \frac{  \left( \sum_{k=1}^{n-1} a_k^{p} \right) \ + \ 
                         \left[   \sqrt[p] {  a_{n}^{p} \ + \ a_{n+1}^{p}   }  \right]^{p} }
                           {  \left( \sum_{k=1}^{n-1} b_k^{p} \right) \ + \ 
                          \left[ \sqrt[p] { b_{n}^{p} \ + \ b_{n+1}^{p}}  \right]^{p}  }  }     \\  
             \ \  < \ & \  \left( \sum_{k=1}^{n-1} \frac {a_k}{b_k}  \right) \ + \    
             \frac{ \sqrt[p] { a_{n}^{p} \ + \ a_{n+1}^{p}}} { \sqrt[p] { b_{n}^{p} \ + \ b_{n+1}^{p}}}    \\ 
               \ \  < \ & \   \left( \sum_{k=1}^{n-1} \frac {a_k}{b_k}  \right) \ + \ 
               \frac{a_n}{b_n} \ + \    \frac{a_{n+1}}{b_{n+1}} \ \ .       
    \end{align*}     
    This was the induction step \ $ n \rightarrow n + 1 $, \ and the inequality of  Theorem \eqref{thesecondone} 
          is proven.
                         
      To complete the proof we have to consider the cases
          $ p = \infty, \ p = -\infty$ \ and \ $ p = 0 $. Because the inequality is proven for real numbers \ 
          $p \neq 0$,   it should be valid also for the limits  $ p = \infty, \ p = -\infty$ \ and \ $ p = 0 $. 
          But we prefer to compute these three cases. \
          Finally, we say something about the statement that `the inequality is sharp'.  
                                
    For an $n$-tuple \ $ \vec{a} = (a_1, a_2, \ \ldots , \ a_n) \in \mathbb{R}^{n}$ \ of positive numbers let \\
    \centerline   {$ A := \max\{ a_1, a_2, \ \ldots \ , \ a_n \} $ \ \ and \ \ 
                   $ a := \min\{ a_1, a_2, \ \ldots \ , \ a_n \} $. }  \\
    The following limits  are well known and easy to proof. 
    $$   \lim_{p\rightarrow +\infty} ( \|\vec{a}\|_p ) = A \ , \  \text{ and }  \ \                                   
         \lim_{p\rightarrow -\infty} ( \|\vec{a}\|_p ) = a   \ .                      $$    
     For  $n$-tuples $ \vec{a}$ and $ \vec{b} $  (without restriction of generality) we choose the arrangements   \\ 
    \centerline{ $  0 < a = a_1 \leq  a_2 \leq a_3 \leq \ \ldots \ \leq a_n = A \ \  \text{ and } \ \ 
                             0 < b := b_1 \leq b_2 \leq b_3 \leq \ \ldots \ \leq b_n =: B $.  }      \\
    It follows very easily that the  limits are   
    $$   \lim_{p\rightarrow +\infty} \ \ \left( \frac{\|\vec{a}\|_p}{\|\vec{b}\|_p} \right) \ \ = \ \
          \frac{A}{B} \ \  < \ \  \sum_{k=1}^{n} \ \frac{a_k}{b_k} \ \quad \text{and} \ \quad 
         \lim_{p\rightarrow -\infty} \ \ \left( \frac{\|\vec{a}\|_p}{\|\vec{b}\|_p} \right) \ \ = \ \
          \frac{a}{b} \ \  < \ \  \sum_{k=1}^{n} \ \frac{a_k}{b_k} \ .          $$     
    The  case $p = 0 $  needs more attention. %%  The result will be used in the next section. 
                    
    We abbreviate the  left hand side of the inequality in Theorem \eqref{thesecondone} by                                    $$    { \mathsf LHS}_p  \ := \   \frac{\|\vec{a}\|_p}{\|\vec{b}\|_p} \ = \
            \ \frac{\sqrt[p]{a_1^{p} + \ a_2^{p} + \ a_3^{p} + \ \ldots \ + \ a_n^{p}}}
          { \sqrt[p]{b_1^{p} + \ b_2^{p} + \ b_3^{p} + \ \ldots \ + \ b_n^{p} } }             $$ 
     for arbitrary $n$-tuples \ $ \vec{a} $ and  $ \vec{b} $ of positive numbers. \
           Instead of \ ${\mathsf LHS}_p$ we consider $ \log({\mathsf LHS}_p)$. If we set \ $ p := 0 $ \ in 
           $ \log({\mathsf LHS}_p)$  we would have a numerator and a denominator $0$, hence 
           we can use the  rules of L'Hospital.
           We compute the limit   $ \log({\mathsf LHS}_p)$  for $ p \rightarrow 0 $, and we calculate 
           with some effort
      $$       \lim_{ p \rightarrow 0} (\log[{\mathsf LHS}_p]) \ = \ 
               \lim_{ p \rightarrow 0} \left( \frac{1}{p} \cdot 
               \log \left[ \ \frac{ a_1^{p} + \ a_2^{p} + \ \ldots \ + \ a_n^{p} }
               { b_1^{p} + \ b_2^{p} + \ \ldots \ + \ b_n^{p} }    \right] \right) 
               \ = \  \frac{1}{n} \cdot  \sum_{k=1}^{n} \log( a_k ) - \log( b_k ) \ .   $$
       We have \ $ \lim_{p \rightarrow 0} \ [{\mathsf LHS}_p] 
               \ = \ \lim_{p \rightarrow 0} \ [ \exp( \log({\mathsf LHS}_p )) ] 
               \ = \ \exp \left(  \lim_{p \rightarrow 0} \ [ \log({\mathsf LHS}_p)]  \right) \ $. 
     Hence we get the limit \        
    $$      \lim_{p\rightarrow 0} \ \ \left( \frac{\|\vec{a}\|_p}{\|\vec{b}\|_p} \right) \ = \
            \lim_{p \rightarrow 0} \ [{\mathsf LHS}_p]  \ = \ 
            \exp \left( \frac{1}{n} \cdot  \sum_{k=1}^{n} \log( a_k ) - \log( b_k ) \right)
            \ = \  \sqrt[n]{ \prod_{k=1}^{n} \ \frac{a_k}{b_k}  }  \ \ .                  $$
    Since there is the well-known inequality between the geometric and the arithmetic mean 
            we finally get the desired inequality
    \begin{align}        
              \lim_{p\rightarrow 0} \ \ \left( \frac{\|\vec{a}\|_p}{\|\vec{b}\|_p} \right) \ = \
              \sqrt[n]{ \prod_{k=1}^{n} \ \frac{a_k}{b_k}  }  \ \leq \  
              \frac{1}{n} \cdot \left( \sum_{k=1}^{n} \frac {a_k}{b_k}  \right) \  < \   
              \left( \sum_{k=1}^{n} \frac {a_k}{b_k}  \right) \ ,       
    \end{align}          
    and the case  $ p=0$ of our main Theorem \eqref{thesecondone} is proven.    
     
    The last point we have to discuss is the remark in Theorem \eqref{thesecondone} that  
    `in this generality it can not be improved'. It means that there are special cases such that
    instead of an inequation we almost have an equation.     \\
    Let $ p $ be any positve number, and let   \\
    \centerline{ $ b_1 = b_2 = b_3 = \ \ldots \ = b_n := 1 $, \ \text{and also} \ $ a_n := 1 $.  }
     Further define \ %  \\  \centerline{ 
     $ a_1 = a_2 = \ \ldots \ = a_{n-1} :=  \frac{1}{p} \: , \ \text{and the tuples} \ \
                 \vec{a} := (a_1,a_2, \ldots , a_n) \ \ \text{and} \ \ \vec{b} := (b_1,b_2, \ldots , b_n)$.  \\ 
    We have just provided the proof of the inequality   
   \begin{align}    \label{inequality soundsoviel} 
            \frac{\|\vec{a}\|_p}{\|\vec{b}\|_p} \ \  <  \ \  \sum_{k=1}^{n} \ \frac{a_k}{b_k} 
            \ = \  \left( \sum_{k=1}^{n-1} \ \frac{1}{p} \right) +  \frac{1}{1}  \ = \  \frac{n-1}{p} + 1 \; .  
   \end{align}              
     For all \ $p \geq 1 $ \ we have \
           $ \max \{ a_1, a_2 , \ldots, a_n \} = 1 =  \max \{ b_1, b_2 , \ldots, b_n \} $, \ and we get the limits 
           $$ \lim_{p\rightarrow +\infty} \ \left( \|\vec{a}\|_p \right) = 1 \ \ \text{and} \ \
              \lim_{p\rightarrow +\infty} \left( \|\vec{b}\|_p \right) = 1    \ . $$  
     Hence both the left hand side and the right hand side of Inequality \eqref{inequality soundsoviel}
            converge to the number $ 1 $ if \ $ p $ converges to infinity.                
            It means that the inequality  converges into an equality. Therefore the inequality in our main
            theorem \eqref{thesecondone} can not be improved by insertion of a constant factor less than $1$.  \\  
            Finally, the last point of Theorem \eqref{thesecondone} has been discussed and the proof is finished. 
 \end{proof} 
 %%%%%%%%%%%%%%%%%%%%%%%%%%%%%%%%%%%%%%%%%%%%%%%%%%%%%%%%%%%%%%%%%%%%%%%%%%%%%%%%%%%%%%%%%%%%%%%%%%
 %%\newpage
  We formulate Theorem \eqref{thesecondone} for infinite sequences. 
   %%%%%%%%%%%%%%%%%%%%%%%%%%%%%%%%%%%%%%%%%%%%%%%%%%%%%%%%%%%%%%%%%%%%%%%%%%%%%%%%%%%%%%%%%%%%%%%%
%   \newpage
   %%%%%%%%%%%%%%%%%%%%%%%%%%%%%%%%%%%%%%%%%%%%%%%%%%%%%%%%%%%%%%%%%%%%%%%%%%%%%%%%%%%%%%%%%%%%%%%%
   \begin{corollary} 
   Let us assume two infinite sequences \ (convergent or not)  \\
      \centerline{ $ a_1, a_2, a_3, \ \ldots , \ a_n , a_{n+1} , a_{n+2} , \ \ldots \ldots  \qquad \text{and} \quad 
            b_1, b_2, b_3, \ \ldots , \ b_n , b_{n+1} , b_{n+2} , \ \ldots \ldots    $   }   \\ 
        of positive numbers.  We have the following inequality for all real numbers $ p \neq 0 $   
        $$  a_1^{p} + a_2^{p} + a_3^{p} + \ \ldots \ + a_n^{p} + \ \ldots \ldots \ \ \  \leq \ \ \  
            \left( b_1^{p} + b_2^{p} + b_3^{p} + \ \ldots \ + \ b_n^{p} + \ \ldots \ldots \right) 
            \ \cdot \  \left( \sum_{k=1}^{\infty} \ \frac{a_k}{b_k} \right)^{p}  \  .       $$
    \end{corollary}
  %  \begin{proof}  This follows immediately from Theorem  \eqref{thesecondone}.  
  %  \end{proof} 
    The corollary means among other things that if the left hand side converges to infinity,
            the right hand side must do the same. 
    %       Or both sides are finite.      $ $  \\  
 % \newpage  
    \section{Examples}  
     In the remark after Inequality  \eqref{inequality soundsoviel} we showed that the inequality 
           of our  Theorem  \eqref{thesecondone} 
           can converge into an equality. Here we add a further example. A third example shows that the difference
           between both sides of the main inequality can become arbitrary big.  
                                                                                                              
     Let $ n $ be a natural number, $n \geq 2$. \  %  Let $K$ be a natural number. 
           We define $n$-tuples of positive numbers for each natural number $ K $, let 
           $\vec{a}_K := (a_1,a_2, \ldots , a_n) \ \text{and} \ \vec{b}_K := (b_1,b_2, \ldots , b_n)$. \
           and we define for every number $K \in \mathbbm{N}$  \\
     \centerline{ $ a_1 = a_2 = \ \ldots \ = a_{n-1} :=  10^{-(2 \cdot K)} \: , \quad \text{and} \quad 
                     b_1 = b_2 = \ \ldots \ = b_{n-1} :=  10^{- K} $,  } \\   
           and let \ $ a_n = b_n := 1 $. With our main inequality we get for each $ p \neq 0$ 
   $$     \sqrt[p]{ \frac{ (n-1) \cdot  10^{-(p \cdot 2 \cdot K)} + 1 }{ (n-1) \cdot  10^{-(p \cdot K)} + 1 } } \ = \  
          \sqrt[p]{ \frac{ a_1^{p} + \ a_2^{p} + \ \ldots \ + \ a_n^{p} }
          { b_1^{p} + \ b_2^{p} + \ \ldots \ + \ b_n^{p} } } \ \  <  \ \ 
          \sum_{k=1}^{n} \ \frac{a_k}{b_k} \ = \  
           \frac{n-1}{ 10^{K}}  +  1   \ .               $$   
          We fix any  \ $ p > 0$.  \  If \ $ K $ converges to infinity we get the case that both sides of the
          inequality  converge to the constant $1$.            
 %%%%%%%%%%%%%%%%%%%%%%%%%%%%%%%%%%%%%%%%%%%%%%%%%%%%%%%%%%%%%%%%%%%%%%%%%%%%%%%%%%%%%%%%%%%%%%%%%  
                                                                 
      We change the parts of   $\vec{a}_K $ and   $\vec{b}_K$, and we define   \\     
      \centerline{ $ a_1 = a_2 = \ \ldots \ = a_{n-1} :=  10^{-K} \: , \quad \text{and} \quad 
                     b_1 = b_2 = \ \ldots \ = b_{n-1} :=  10^{- (2 \cdot K)} $ ,  } \\   
           and let \ $ a_n = b_n := 1 $  as before.  In this case we get for each $ p \neq 0$ 
   $$     \sqrt[p]{ \frac{ (n-1) \cdot  10^{-(p \cdot K)} + 1 }{ (n-1) \cdot  10^{-(p \cdot 2 \cdot K)} + 1 } } \ = \  
          \sqrt[p]{ \frac{ a_1^{p} + \ a_2^{p} + \ \ldots \ + \ a_n^{p} }
          { b_1^{p} + \ b_2^{p} + \ \ldots \ + \ b_n^{p} } } \ \  <  \ \ 
          \sum_{k=1}^{n} \ \frac{a_k}{b_k} \ = \  
          (n-1) \cdot 10^{K} \ + \ 1   \ .               $$   
          We fix \ $ p > 0$. \  If \ $ K $ converges to infinity we see that the left hand side converges to
          the constant $1$,  while the right hand side  converges to infinity.             
 { $ $ }   \\ 
{ \bf Acknowledgements: }   \rm  We wish to thank  Stefan Th\"urey for careful reading of the paper.
%     Acknowledgements: \ We wish to thank the unknown referee  for careful reading of the paper and
%                        suggesting many improvements.  
%\thanks{This research is supported by ...\par This paper is a lecture that was given at...} 
%   Optional. Only one command thanks is allowed, use \par inside text if you need multiple thanks.  
 \\
 %%   { \bf Acknowledgements: }   \rm 
% 

      $ $     \\   \\  
\author{ \centerline{VOLKER W. TH\"UREY,}  \\  
         \centerline{Rheinstr. 91, \ 28199 Bremen, Germany, \ T: 49 (0)421/591777,}  \\
         \centerline{e-mail: \  volker@thuerey.de } }        %    $ { } $  \hfill  \
%        Editorial office suggest use of the AMS commands 
%        (which can be cited like ordinary bibitem references):
%
%\Refs
%\bibitem{HLP} \by  $ \rm G.H. Hardy, J.E. Littlewood, G. P\acute{o}lya $
%              \book Inequalities
%              \publ Cambridge University Press, 1934
%              \year 1934
%\endref              
%\bibitem{AS2}
%        \by M. Abramowitz and I. A. Stegun (Eds)
%        \book Handbook of Mathematical Functions with Formulas, Graphs, and Mathematical Tables
%        \publ National Bureau of Standards, Applied Mathematics Series {\bf 55}, 9th printing
%        \publaddr Washington 
%        \year 1970
%\endref
%  
%\bibitem{Burn}
%        \by W. Burnside
%        \paper A rapidly convergent series for $\log N!$
%        \jour Messenger Math
%        \vol 46
%        \issue 1
%        \yr 1917
%        \pages 157--159
%\endref

%\endRefs
\end{document}